\documentclass[a4paper,12pt]{article}
\usepackage{amsmath,amsthm,amsfonts,amssymb,amscd} 
\usepackage{tikz}
\usepackage{color}
\newtheorem{thm}{Theorem}[section]
\newtheorem{cor}[thm]{Corollary}
\newtheorem{lem}[thm]{Lemma}
\newtheorem{prop}[thm]{Proposition}

\theoremstyle{defn}
\newtheorem{defn}{Definition}[section]
\newtheorem{rem}{Remark}[section]

\newcommand{\eval}[2][\right]{\relax
  \ifx#1\right\relax \left.\fi#2#1\rvert}

\theoremstyle{plain}
\title{Rigidity for partially hyperbolic diffeomorphisms}
\author{R\'egis Var\~ao}

\newcommand{\tilc}{\til{\mathcal F}^c}
\newcommand{\til}{\widetilde}
\usepackage{color}

\begin{document}

\maketitle

\begin{abstract}
In this work we completely classify $C^\infty$ conjugacy for conservative partially hyperbolic diffeomorphisms homotopic to a linear Anosov automorphism on the 3-torus by its center foliation behavior. We prove that the uniform version of absolute continuity for the center foliation is the natural hypothesis to obtain $C^\infty$ conjugacy to its linear Anosov automorphism. On a recent work Avila, Viana and Wilkinson proved that for a perturbation in the volume preserving case of the time-one map of an Anosov flow absolute continuity of the center foliation implies smooth rigidity. The absolute version of absolute continuity is the appropriate sceneario for our context since it is not possible to obtain an analogous result of Avila, Viana and Wilkinson for our class of maps, for absolute continuity alone fails miserably to imply smooth rigidity for our class of maps. Our theorem is a global rigidity result as we do not assume the diffeomorphism to be at some distance from the linear Anosov automorphism. We also do not assume ergodicity. In particular a metric condition on the center foliation implies ergodicity and $C^\infty$ center foliation.
\end{abstract}
\maketitle

\section{Introduction}

A fundamental problem in Dynamical Systems is the description of the orbits of elements under a given motion law, in our case this law will be given by diffeomorphisms. One of the most effective ways to do so is through classification, that is understanding a given dynamics by its relation with some other better known model.
A typical and important class of examples of Dynamical Systems are the hyperbolic diffeomorphisms, also known as Anosov diffeomorphisms. Anosov \cite{Anosov} proved that these maps (here we always assume at least $C^{1 + \alpha}$) are ergodic with respect to the volume measure. One of the key steps of his proof relies on the understanding of the behavior of the two associate invariant foliations, the stable and the unstable, which are uniformly contracting and uniformly expanding respectively.

These hyperbolic maps exhibit enough structure so that one can analyze them using some geometric and metric tools. A diffeomorphism $f:M \rightarrow M$ on a smooth manifold $M$ is said to be transitive if there exists $x \in M$ such that $\{f^n(x)\}_n$ is dense in $M$. We say that $f$ is robustly transitive if there exists a $C^1$ neighborhood of $f$ such that all diffeomorphisms in this neighborhood are transitive. One of the first examples of robustly transitive systems were the Anosov diffeomorphisms (see \cite{BDV} and references therein). It was then thought that hyperbolic diffeomorphisms were the only robustly transitive systems, in fact Ma\~n\'e \cite{mane.ECL} proved that $C^1$ robustly transitive diffeomorphims on a surface are the hyperbolic ones. But Shub \cite{shub.T4} and Ma\~n\'e \cite{mane.stability.conjecure} gave examples of robustly transitive systems which were not hyperbolic. Their systems fall on a larger class known as Partially Hyperbolic diffeomorphisms.

The partially hyperbolic diffeomorphisms are similar to Anosov diffeomorphism as they have expanding and contracting directions and the respectively associated stable and unstable foliations. The difference is that there is a center direction for a partially hyperbolic diffeomorphism which is dominated by the stable and unstable directions (see Definition \ref{defi:PH}). A major difficulty to understand these systems comes on predicting the behavior of the center direction since it may not exhibit any kind of hyperbolicity. And as one expects this brings new difficulties, but most important it gives rise to new dynamical behaviors.


The stable and unstable foliations of a partially hyperbolic diffeomophism are absolutely continuous, this means that for a set of full volume almost every leaf intersects this set in a full leaf-volume measure. That is, one should see absolute continuity of a foliation as a Fubini like behavior, which may also be seen as some sort of regularity condition of these foliations. In contrast, the center foliation may exhibit a new type of behavior, it may have atomic disintegration: there is a set of full volume which intersects each center leaf in a finite number of points. At first this might sound as a pathological behavior, but it turns out that this is in fact a common behavior for the center foliation \cite{AVWI,PTV,ruelle.wilkinson}. 

If one imposes some regularity on a center foliation, for instance absolute continuity as it happens for the stable and unstable foliations, one may get a very rigid structure. For instance Avila, Viana and Wilkinson \cite{AVWI} assuming absolute continuity of the center foliation for volume preserving perturbations of a time-one map of the geodesic flow on a constant negatively curved surface, they obtain that in fact this perturbation is a time-one map of a $C^\infty$ Anosov flow. 

%


So far the main results on the rigidity problem concerning the center foliation are the ones presented in \cite{AVWI}, as described above, and in \cite{varao.etds}. On \cite{varao.etds} it is studied the rigidity problem for the class of derived from Anosov systems, which are partially hyperbolic diffeomorphisms homotopic to a linear Anosov. These systems have a hyperbolic memory, more precisely they are semi-conjugate to a linear Anosov, also called as the linearization of this diffeomorphism (see \S \ref{sec:preliminaries} for details).

It is assumed on \cite{varao.etds} that the center foliation is $C^1$ and have transversely absolutely continuous foliations with bounded Jacobian, and the implication is $C^{1}$ conjugacy with a linear Anosov. Our main result largely extends this last result since it does not assume any smoothness from the center foliation, in fact it assumes an almost everywhere condition and, then, it implies a much stronger regularity, $C^\infty$. 

The low regularity assumption in our result is simply a uniform version of absolute continuity of a foliation. More precisely, the disintegration of volume on any foliated box has conditional measures with a uniformly bounded density with respect to the leaf-volume (see Definition \ref{defi:UBD}). Micena and Tahzibi \cite{micena.tahzibi.UBD} called this the UBD Property, our main result also answers their conjecture that UBD property would imply $C^1$ conjugacy, we obtain in fact $C^\infty$.

Transversely absolutely continuous foliations with bounded Jacobian means that the holonomies of transversal foliations have uniformly bounded Jacobian, one can easily check that this implies the UBD property (the proof follows from the appropriate modifications of  Proposition 6.2.2 of \cite{brin.stuck}).

In the context of Avila, Viana, Wilkinson \cite{AVWI} absolute continuity implies smoothness, in our context this is far from possible and one cannot drop the UBD property as, by \cite[Theorem 1.3]{varao.etds}, the center foliation may be $C^1$ and $f$ not even $C^1$ smoothly conjugate to its linearization. We point out that being a $C^1$ foliation implies to have locally uniform bounded densities. Hence our hypothesis (the uniform version of absolute continuity) is a natural one and, from the explained before, can be seen as a sharp condition. Our main result, Theorem \ref{main:rigidity.C}, in contrast to other results is a global rigidity result, i.e. one does not assume any proximity of the derived from Anosov to its linearization. We also point out that we do not assume ergodicity.

We now state our main result:
\begin{thm}\label{main:rigidity.C}
Let $f:\mathbb{T}^3 \rightarrow \mathbb T^3$ be a conservative partially hyperbolic diffeomorphism homotopic to a linear Anosov. The center foliation has the Uniform Bounded Density property if and only if  $f$ is $C^{\infty}$ conjugate to its linearization.
\end{thm} 


\textit{Structure of the paper:} The next section is the preliminaries where one finds the formal definitions and results mentioned in this section and elsewhere. After the preliminaries we prove our main result. Remark \ref{rem}, where we propose a general program to tackle rigidity problems for the center foliations of partially hyperbolic diffeomorphisms in a more general setting, comes after the proof of Theorem \ref{main:rigidity.C}.


\section{Preliminaries} \label{sec:preliminaries}

We briefly present the main definitions and results we will be using throughout this work.
 
\begin{defn}\label{defi:PH}
Let $M$ be a smooth compact Riemannian manifold. A diffeomorphism $f:M \rightarrow M$ is called partially hyperbolic if the tangent bundle of the ambient manifold  admits an invariant decomposition $TM = E^s \oplus E^c \oplus E^u$, such that for all unit vectors $v^{\sigma} \in E^{\sigma}, \sigma \in \{s, c, u\}$ and all $x, y, z \in M$
\[
 \|D_xf v^s \| < \|D_yf v^c\| < \| D_zf v^u\|
\] and
$\|D_xf v^s\| < 1 < \|D_zf v^u\|$.
\end{defn}

It is well known that for partially hyperbolic diffeomorphisms, there are foliations $\mathcal F^{s}$ and $\mathcal F^{u}$ tangent to the sub-bundles $E^{s}$ and $E^{u}$ called \textit{stable} and \textit{unstable foliation} respectively (for more details see for example \cite{YP}). On the other hand, a priori there is no garantee of the existence of a center foliation for a partially hyperbolic diffeomorphism. But that is not our concern since in the context we will be working (partially hyperbolic diffeomorphisms on $\mathbb T^3$)  all partially hyperbolic diffeomorphisms (as defined above) on $\mathbb T^3$ admit a central foliation tangent to $E^c$ by a result of Brin, Burago, Ivanov \cite{BBI}. We say that $f$ is an Anosov diffeomorphism if it satisfies Definition \ref{defi:PH} but there is no center direction. Anosov diffeomorphisms also have stable and unstable invariant foliations as described above.

\begin{defn} \label{defi:DA}
We say that $f\colon \mathbb{T}^n \rightarrow \mathbb{T}^n$ is a derived from Anosov diffeomorphism or just a DA diffeomorphism if it is a partially hyperbolic diffeomorphism homotopic to a linear Anosov automorphism $A:\mathbb T^n \rightarrow \mathbb T^n$.
\end{defn} 

If $f$ is a DA diffeomorphism we call the linear Anosov $A$ as the linearization of $f$, it also means that $f$ has a hyperbolic memory. More precisely, by results of J. Franks \cite{Franks} and A. Manning \cite{Manning} there is a semi-conjugacy $h: \mathbb{T}^3 \rightarrow \mathbb{T}^3$, which we will call  the Franks-Manning semi-conjugacy, between $f$ and its linearization $A$, that is,
\begin{equation} \label{semiconjugacy}
A \circ h = h \circ f.
\end{equation}
Moreover, this semi-conjugacy has the property that there exists a constant $\Omega  \in \mathbb{R}$ such that  if $H, F : \mathbb{R}^3 \rightarrow \mathbb{R}^3 $ denotes the lift of $h$ and $f$ to $\mathbb{R}^3$ respectively, we have $\|H(x) - x\| \leq \Omega$ for all $x \in \mathbb{R}^3$, and given two points $a,b \in \mathbb R^3$
\begin{eqnarray}\label{eq:h}
H(a) = H(b) \Leftrightarrow \| F^n(a) - F^n(b)\| < \Omega , \forall n\in \mathbb Z. 
\end{eqnarray}

Let $(M, \mu, \mathcal B)$ be a probability space, where $M$ is a compact metric space, $\mu$ a probability measure and $\mathcal B$ the Borelian $\sigma$-algebra.
Given a partition $\mathcal P$ of $M$ by measurable sets, we construct a probability space $(\mathcal P, \widetilde \mu, \widetilde{\mathcal B})$ in the following way. Let $\pi:M \rightarrow \mathcal P$ be the canonical projection, that is, $\pi$ associates to a point $x$ of $M$ the partition element of $\mathcal P$ that contains it. Then we define $\widetilde \mu := \pi_* \mu$ and $ \widetilde{\mathcal B}:= \pi_*\mathcal B$.

\begin{defn} \label{definition:conditionalmeasure}
 Given a partition $\mathcal P$. A family $\{\mu_P\}_{P \in \mathcal P} $ is a \textit{system of conditional measures} for $\mu$ (with respect to $\mathcal P$) if
\begin{itemize}
 \item[i)] given $\phi \in C^0(M)$, then $P \mapsto \int \phi d\mu_P$ is measurable;
\item[ii)] $\mu_P(P)=1$ $\widetilde \mu$-a.e.s;
\item[iii)] if $\phi \in C^0(M)$, then $\displaystyle{ \int_M \phi d\mu = \int_{\mathcal P}\left(\int_P \phi d\mu_P \right)daq }$.
\end{itemize}
\end{defn}

When it is clear which partition we are referring to, we say that the family $\{\mu_P\}$ \textit{disintegrates} the measure $\mu$.  

\begin{prop}\label{prop:unique.disintegration} \cite{EW,Ro52}
 Given a partition $\mathcal P$, if $\{\mu_P\}$ and $\{\nu_P\}$ are conditional measures that disintegrate $\mu$ on $\mathcal P$, then $\mu_P = \nu_P$ $\widetilde \mu$-a.e..
\end{prop}

An easy consequence of this proposition, but which will be useful for us, is

\begin{cor} \label{cor:same.disintegration}
 If $T:M \rightarrow M$ preserves a probability $\mu$ and the partition $\mathcal P$, then  $T_*\mu_P = \mu_{T(P)}, \widetilde \mu$-a.e..
\end{cor}
\begin{proof}
 It follows from the fact that $\{T_*\mu_P\}_{P \in \mathcal P}$ is also a disintegration of $\mu$.\end{proof}

\begin{defn} \label{def:mensuravel}
We say that a partition $\mathcal P$ is a measurable partition (or countably generated) with respect to $\mu$ if there exist a measurable family $\{A_i\}_{i \in \mathbb N}$ and a measurable set $Y$ of full measure such that 
if $B \in \mathcal P$, then there exists a sequence $\{B_i\}$, where $B_i \in \{A_i, A_i^c \}$ such that $B \cap Y = \bigcap_i B_i \cap Y$.
\end{defn}

The next theorem guarantees the existence of conditional measures with respect to a measurable partition.

\begin{thm}[Rokhlin's disintegration \cite{Ro52}] \label{teo:rokhlin} 
 Let $\mathcal P$ be a measurable partition of a compact metric space $M$ and $\mu$ a probability. Then there exists a disintegration by conditional measures for $\mu$.
\end{thm}

Recall that a foliation is absolutely continuous if the conditional measures on the leaves of the foliation are equivalent to the Lebesgue measure. In other words, a set of full volume intersects almost every leaf in a set of full leaf-volume. The Uniformly Bounded Density property is the uniform version of a foliation being absolutely continuous. 

\begin{defn}\label{defi:UBD}
 We say that a foliation $\mathcal F$ has the Uniform Bounded Density property (or UBD for short) if there exists a uniform constant $K$ such that for any foliated box $\mathcal B$
 $$K^{-1} \leq \frac{dm_x^\mathcal B}{d \widehat{\lambda_x}} \leq K $$ 
 where $m_x^{\mathcal B}$ is the conditional measure of volume on the segments of $\mathcal F$ on $\mathcal B$ and $\widehat{\lambda_x}$ means the normalized Lebesgue measure on the connected component of $\mathcal F(x) \cap \mathcal B$ which contains $x$.
\end{defn}

\section{The proof of Theorem \ref{main:rigidity.C}.}

\begin{proof}

We work on the universal cover of $\mathbb T^3$, $\mathbb R^3$. Let $\pi:\mathbb R^3 \rightarrow \mathbb T^3$ be the canonical projection. We denote by $F$ and $H$ the lifted function of $f$ and the semi-conjugacy $h$ respectively. And for the other lifts we use the symbol $\sim$, for instance $\tilc$ is the lift of the center foliation $\mathcal F^c$. We shall use $\tilc$ to refer to the center foliation of $F$ and $\tilc_x, \tilc(x)$ for the lifted center foliation through $x$. Suppose the splitting for $A$ is of the form $E^{ss} \oplus E^{wu} \oplus E^{uu}$, otherwise work with $F^{-1}$. Throughout the proof we shall denote by $\lambda$ the eigenvalue of the center direction of $A$, which is greater than one since we suppose $A$ to have an expanding center direction. We also see $A$ as a partially hyperbolic with the splitting of the tangent bundle as $E^s_A \oplus E^c_A \oplus E^u_A$.

At the universal cover let us define a ``base space" 
$$B : = \bigcup_{x \in \til{\mathcal F}^s(0)} \til{\mathcal F}^u(x).$$ 
By \cite{hammerlindl.thesis} all the center leaves of $F$ intersect $B$ at a unique point. We may also assume that $0$ is a fixed point for $F$, otherwise take a periodic point $p$ of period $n$ for $f$ and work with $f^n$ instead of $f$. The lifted function $F$ is such that $\pi(0)=p$. This implies that $B$ is $F$ invariant, $F(B)=B$. On $\mathbb R^3$ the center foliation is an oriented foliation. We also assume that $F$ preserves the center foliation orientation, otherwise work with $f^2$ instead of $f$ (and the respective lift). Note that these changes do not affect the result since $h$ is also a semiconjugacy between $f^n$ and $A^n$ for all $n \in \mathbb N$.
%

Let us define the following set
$$ \mathcal B_0 = \{y \in \mathbb R^3 \; | \; d^c(P(y),y) \leq \Omega \text{ and } y \in \til{ \mathcal F}^{c,+}\}$$
where $d^c(.,.)$ is the distance inside the center leaf, $\mathcal F^{c,+}$ stands for the positive half side of the orientable foliation $\til{\mathcal F}^c$ and $\Omega$ as defined in \S \ref{sec:preliminaries}.

We consider the following iterations $\mathcal B_k := F^k(\mathcal B_0)$ for all $k \in \mathbb N$. Now let $m_{x}^{\mathcal B_k}$ be the conditional measure defined on $\til{\mathcal F}^c_x \cap \mathcal B_k$ which is the disintegration of volume restricted to $\mathcal B_k$ and the partition given by $\{\til{\mathcal F}^c_{\xi} \cap \mathcal B_k \}_{\xi \in B}$. In particular $m_{x}^{\mathcal B_k}$ is a probability measure. The existence of $m_x^{\mathcal B_k}$ is guaranteed by Theorem \ref{teo:rokhlin}, but there are two remarks concerning the application of Rokhlin's theorem, the first one is that Theorem \ref{teo:rokhlin} is for a probability measure, the second is that $\til{\mathcal F}^c_x \cap \mathcal B_k$ should be checked to be a measurable partition. Let us see how to deal with these two issues. Note that it is not difficult to see that $B$ is a metrizable space with the topology of subspace of $\mathbb R^3$. Hence let $Z \subset B$ be a ball in $B$ and consider $Z_k:=(P|\mathcal B_k)^{-1}(Z)$, let $S_Z$ be a countable dense set in $Z$ and for all discs $D(s,r)$ inside $B$ of center $s \in S_Z$ and radius $r \in \mathbb Q$ let $A_{s,r}:=  (P|\mathcal B_k)^{-1}(D(s,r))$, then these $A_{s,r}$ (there are a countable number of them) play the rule of $A_i$ in Definition \ref{def:mensuravel} (and on Definition \ref{def:mensuravel} one may assume $Y$ to be $ \mathbb R^3$). Hence we may apply Theorem \ref{teo:rokhlin} to disintegrate $Vol|Z_k$ when normalized. Then the conditional measures $m_x^{\mathcal B_k}$ comes from this disintegration. But the normalization does not affects the disintegration as since the only change one gets is for the projected measure (which in the notation of Definition \ref{definition:conditionalmeasure} would be $\widetilde \mu$).



Let us denote $\mu_{\mathcal B_k}:=P_*Vol|\mathcal B_k$. That is, $\mu_{\mathcal B_k}$ is the projection on $B$ of the volume measure restricted to $\mathcal B_k$. We are able to relate two different conditional measures, by \cite[Lemma 3.2]{AVWI} we have that
\begin{eqnarray}\label{eq:two.conditionals}
m_x^{\mathcal B_k} \frac{d \mu _{\mathcal B_k}}{d \mu_{\mathcal B_0}}(x)=
m_x^{\mathcal B_0}. 
\end{eqnarray}

Note that these measures differ up to multiplication by the constant $\frac{d \mu _{\mathcal B_k}}{d \mu_{\mathcal B_0}}(x)$. Observe as well that
 $$ m_x^{\mathcal B_k}(.) = \rho_k(x,.) \widehat \lambda_{\mathcal B_k}(x,.), \; \forall k \in \mathbb N,$$
where $\widehat{\lambda}_{\mathcal B_k}(x,.)$ is the normalized length measure on the leaf $\tilc_x \cap \mathcal B_k$ and $\rho_k(x,.):\tilc_x \cap \mathcal B_k \rightarrow \mathbb R$ is the density function. By the UBD property of the center foliation $\rho_k \in [K^{-1},K]$. Then
\begin{eqnarray*}
m_x^{\mathcal B_k}(I) & = & \int_I \rho_k(x,\zeta) d\widehat \lambda_{\mathcal B_k}(x,.)(\zeta) = \int_I \rho_k(x,\zeta) d\frac{ \lambda_{\mathcal B_k}(x,.)}{\lambda_x(\tilc_x \cap \mathcal B_k)}(\zeta)\\ & \leq & \frac{K}{\lambda_x(\tilc_x \cap \mathcal B_k)} \int_I d  \lambda_{\mathcal B_k}(x,.)(\zeta) \leq \frac{K \lambda_{\mathcal B_k}(x,I)}{\lambda_x(\tilc_x \cap \mathcal B_k)} 
\end{eqnarray*}
for all $k \in \mathbb N$ and $I \subset \tilc_x$ measurable and where $\lambda_x$ is the Lebesgue (length) measure on the leaf $\tilc$. Analogously
\begin{eqnarray*}
m_x^{\mathcal B_k}(I)\geq \frac{K^{-1} \lambda_{\mathcal B_k}(x,I)}{\lambda_x(\tilc_x \cap \mathcal B_k)}. 
\end{eqnarray*}

From the above inequalities we have

$$ \frac{d \mu _{\mathcal B_k}}{d \mu_{\mathcal B_0}}(x)= \frac{m_x^{\mathcal B_0}(.)}{m_x^{\mathcal B_k}(.)} \leq K^2 \frac{\lambda_x(\mathcal F^c_x \cap \mathcal B_k)}{\lambda_x(\mathcal F^c_x \cap \mathcal B_0)},$$
similarly we get
$$ \frac{d \mu _{\mathcal B_k}}{d \mu_{\mathcal B_0}}(x) \geq K^{-2} \frac{\lambda_x(\mathcal F^c_x \cap \mathcal B_k)}{\lambda_x(\mathcal F^c_x \cap \mathcal B_0)}.$$



So far we have taken for granted that if the conditional measure exists for $\mathcal B_0$ it also exists for $\mathcal B_k$, this is shown in the next lemma.

\begin{lem}
 There is a set $A \subset \mathbb R^3$ which is $F$-invariant, has full volume measure and is $\tilc$ foliated for which $m_{x}^{\mathcal B_k}$ is well defined for all $x \in A$ and $k \in \mathbb N$.
\end{lem}
\begin{proof}
Let $S \subset B$ be a countable dense set of $B$, let $S':=\pi(S)$ and $B':=\pi(B)$ (recall that $\pi:\mathbb R^3 \rightarrow \mathbb T^3$ is the canonical projection). Then $S'$ is a dense set of $B'$. For a given $s \in S'$ and $r \in \mathbb Q$ define $C'(s,r)$ to be a foliated box such that all segments of the center foliation have length $r$ and $Vol(\partial C'(s,r))=0$ and satisfying also that $Vol(\bigcup_{s \in S'} C'(s,r))=1$.

We disintegrate volume on the foliated box $C'(s,r)$, hence its conditional measures are defined on a set $M(s,r)$ which has $Vol|C'(s,r)$ full measure. Then $Vol( M(s,r) \cup C'(s,r)^c )=1$. Now define
$$A_0 := \bigcap_{r \in \mathbb Q} \left(\bigcup_{s \in S'} C'(s,r)\right) \text{ and } A_1:= A_0 \setminus \bigcup_{(s,r)\in S'\times \mathbb Q } \partial \left( C'(s,r)   \right). $$

Notice that $A_0$ and $A_1$ have full volume and that if $x \in A_1$ then there are conditional measures defined on the center leaf of $x$ with arbitrary size. Let $A_2:= \bigcap_{i \in \mathbb Z} f^{i}(A_2)$, which is an $f$ invariant full volume set. Then the lemma is proved by taking $A:=\pi^{-1}(A_2)$. 
 \end{proof}

Let $\eta_{x,k}$ be a measure (not a probability) defined on $\tilc_x\cap \mathcal B_k$ as 
$$ \eta_{x,k}:= \lambda^k m_x^{\mathcal B_k}.$$
Recall that $\lambda$ is the eigenvalue of the center direction of $A$. Now, using Equation (\ref{eq:two.conditionals}), on $\mathcal B_0$ we have

$$\eta_{x,k}= \lambda^k m_x^{\mathcal B_k} = (\frac{d \mu_{\mathcal B_k}}{d \mu_{\mathcal B_0 }}(x))^{-1} \lambda^k m_x^{\mathcal B_0 }$$ 
and from the inequalities above
$$\frac{d \mu_{B_k}}{d \mu_{\mathcal B_0}}(x) = \alpha_{x,k} \frac{\lambda_x(\mathcal F^c_x \cap \mathcal B_k)}{{\lambda_x(\mathcal F^c_x \cap \mathcal B_0)}},$$ where $\alpha_{x,k} \in [K^{-2}, K^2]$, for all $x \in A \subset \mathbb R^3$ and $k \in \mathbb N$.

Combining the above we may rewrite
$$\eta_{x,k} =  \left(\alpha_{x,k} \frac{\lambda_x(\mathcal F^c_x \cap \mathcal B_k)}{{\lambda_x(\mathcal F^c_x \cap \mathcal B_0)}}\right)^{-1} \lambda^k m_x^{\mathcal B_0}.$$

\begin{lem}\label{lemma:growth.bound}
 There is $\beta >0$ such that $\lambda^k / \lambda_x(\til{\mathcal F}^c_x \cap \mathcal B_k) \in [1/\beta, \beta]$ for all $x \in \mathbb R^3$.
\end{lem}
\begin{proof}
Let $x, y \in \mathbb R^3$ be the extreme points of the segment $\tilc_x \cap \mathcal B_0$, then $F^n(x)$ and $F^n(y)$ are the extreme end points of $\tilc(F^n(x)) \cap \mathcal B_k$. Hence we want to measure the growth of these extreme points inside the center foliation. By \cite{hammerlindl.thesis} the center foliation $\tilc$ is quasi-isometric. That means that there exists a constant $Q$ such that for all $z \in \tilc(w)$, then $Q^{-1}||z-w|| \leq d^c(z,w) \leq Q ||z-w||$ where $d^c$ is the distance inside the center foliation. Hence due the quasi-isometry of the center foliation we only need to analyze the growth of $||F^n(x)-F^n(y)||$. We will majorate, to minorate is an analogous argument.
\begin{eqnarray*}
  ||F^n(x)-F^n(y)|| \leq & & ||F^n(x) - H\circ F^n(x)||\\ & + & ||H\circ F^n(x) - H\circ F^n(y)|| + ||H\circ F^n(y) - F^n(y)||.
\end{eqnarray*}
Because $H$ is uniformly close to the identity, as exposed in \S \ref{sec:preliminaries}, the first and third terms of the right hand side are uniformly bounded. We have to see that the term  $||H\circ F^n(x) - H\circ F^n(y)|| = ||A^n(H(x)) - A^n(H(y))||$ grows at rate $\lambda$. On the last identity we used $H\circ F = A \circ H$.

By \cite{ures} the semi-conjugacy sends center leaf into center leaf (i.e. $H(\til{\mathcal F}^c_F)=\til{\mathcal F}^c_A$), therefore $||A^n(H(x)) - A^n(H(y))||$ does grows at rate $\lambda$ since $H(x) \in \tilc_A(H(y))$. 
 \end{proof}

\begin{lem}\label{lemma:eta}
 For all $x \in A$ there exists a measure $\eta_x$  such that $F_*\eta_x=\lambda^{-1}\eta_{F(x)}$ and $\eta_x = \rho_x \lambda_x$ where the density $\rho_x(.)$ is uniformly bounded (independent of $x \in A$).
\end{lem}
\begin{proof}
Since the disintegration is unique and $F$ is volume preserving, that means that the family $\{F_*\eta_{x,k}\}_{x \in B}$ when normalized is a disintegration of volume restricted to $\mathcal B_{k+1}$ and the family $\{\eta_{x,k+1} \}_{x \in B}$ when normalized is a disintegration of volume restricted to $\mathcal B_{k+1}$, hence these normalized measures are the same by Proposition \ref{prop:unique.disintegration}. As a consequence $F_*\eta_{x,k}$ and $\eta_{x,k+1}$ are the same up to multiplication by a constant. To find out this constant one could simply evaluate these measures on the same set. Using the $F$ invariance of $\tilc$ and the definition of $\eta_{x,k}$ we obtain
\begin{eqnarray*}
F_*\eta_{x,k}(F(\mathcal B_k \cap \tilc_x)) & = &  \eta_{x,k}(F^{-1}( F(\mathcal B_k \cap \tilc_x) )) = \eta_{x,k}(\mathcal B_k \cap \tilc_x) \\ 
& = & \lambda^k = \lambda^{-1} \lambda^{k+1} = \lambda^{-1} \eta_{F(x),k+1}(\mathcal B_{k+1} \cap \tilc_{F(x)})\\
& = &\lambda^{-1} \eta_{F(x),k+1}(F(\mathcal B_k \cap \tilc_x)),
\end{eqnarray*}
therefore $F_* \eta_{x,k}= \lambda^{-1} \eta_{F(x),k+1}$.

Hence if there is a subsequence $k_{i(x)}$ such that $\eta_{x,k_{i(x)}}$ converges weakly to $\eta_x$ and $\eta_{F(x),k_{i(x)}+1}$ converges weakly to $\eta_{F(x)}$, then $F_*\eta_x = \lambda^{-1} \eta_{F(x)}$.
 
 We now show how to define these subsequences. By the Axiom of Choice we can choose a set $W \subset A$ such that if $x, y \in C$ and $x \neq y$ then $\{F^{n}(x)\}_{n \in \mathbb Z} \cap \{F^{n}(y)\}_{n \in \mathbb Z} = \emptyset$. For $x\in W$ let $\alpha_{x,k_{i(x)}}$ be a subsequence such that $\{\alpha_{x,k_{i(x)}}\}_{i(x)}$ is a convergent subsequence. Now consider the subsequence $k_{i(x)} +1$ and take $i(F(x))$ a subsequence of $i(x)$ such that $\{ \alpha_{F(x),k_{i(F(x))+1} } \}_{i(F(x))}$ is a convergent subsequence. Repeat the argument for the positive iterates $F^n(x)$. By cantor diagonal process there exists a sequence $j(x)$ such that $\alpha_{F^n(x),(k_{j(x)} +n)_{j(x)}}$ is a convergent subsequence for all $n \in \mathbb N$. Repeat the argument for the backward iteraction and one obtain the desired subsequence in the orbit of $x$.
 \end{proof}

Some comments on the arbitrariness of the set $W$ used in the lemma above may be seen on Remark \ref{rem}, after the proof of our main theorem.

One of our goals is to calculate the center Lyapunov exponent for all points of $\mathbb R^3$. For that we will need the dynamically defined measures $\eta_x$, but they are defined only almost everywhere.  We wish we could define these measures on all center leaves. A priori it is not possible, what we shall do is to construct some sort of fake dynamically defined conditional measures, but they are good enough for us to compute the Lyapunov exponents for all point.

Given $\xi \in \mathbb R^3$ take any sequence $\{\xi_n\} \subset \mathbb R^3$ for which it is defined $\eta_{\xi_n}$ for all $n$ and $\displaystyle \lim_{n \rightarrow \infty} \xi_n = \xi$. Let $\mathcal I_{\xi}$ be the set of all connected intervals on $\mathcal F^c_{\xi}$. Given $I \in \mathcal I_{\xi}$ define
\begin{eqnarray*}
m_\xi (I) &:=& \liminf_{n \rightarrow \infty} \frac{1}{n} \sum_{i=1}^n \eta_{\xi_i}(I_{\xi_i}), \quad \forall I \in \mathcal I_{\xi};\\
M_\xi (I) &:=& \limsup_{n \rightarrow \infty} \frac{1}{n} \sum_{i=1}^n \eta_{\xi_i}(I_{\xi_i}), \quad \forall I \in \mathcal I_{\xi};
\end{eqnarray*}
where $I_{\xi_i} \in \mathcal I_{\xi_i}$ is defined as 
$$I_{\xi_i} := \text{ bounded interval of }\til{\mathcal F}^c_{\xi_i} \backslash \left( \left(\bigcup_{z \in \til{\mathcal F}^{s}(a_I)} \til{\mathcal F}^{u}(z) \right) \; \bigcup \;  \left(\bigcup_{z \in \til{\mathcal F}^{s}(b_I)} \til{\mathcal F}^{u}(z)\right)  \right)$$
where $a_I$ and $b_I$ are the bottom and top extreme points respectively of the interval $I$.  

\begin{lem}\label{lemma:min.max.ineq}
There exists $\gamma >1$ independent of the choice of $\xi$ and $\{\xi_n\}_n$ such that
 \begin{eqnarray*}
  \gamma^{-1} \lambda_{\xi}(I) \leq m_{\xi}(I) \leq M_{\xi}(I) \leq \gamma \lambda_{\xi}(I),
 \end{eqnarray*}
 for all $I \in \mathcal I_\xi$ small enough.
\end{lem}
\begin{proof}
By \cite[Theorem B]{holder.foliations} the unstable holonomy inside center-unstable leaf and the stable holonomy inside the center-stable leaf is uniformly $C^1$. The uniformity is defined for interval $I$ with a uniform bounded length and the holonomies are uniformly $C^1$ if we consider then to be close of some uniformly bounded distance. And from Lemma \ref{lemma:eta} we know that $\eta_{\xi_i}$ have uniformly bounded densities, hence the lemma follows.
 \end{proof}

Since $F(\mathcal I_\xi) = \mathcal I_{F(\xi)}$ it makes sense to state
\begin{lem}
For all $I \in \mathcal I_{F^n(\xi)}$ small enough
 \begin{eqnarray*}
  F_*^n m_{\xi}(I) &=& \lambda^{-n} m_{F^n(\xi)}(I);\\
  F_*^n M_{\xi}(I) &=& \lambda^{-n} M_{F^n(\xi)}(I).
 \end{eqnarray*}
\end{lem}
\begin{proof}
We prove the result for $m_{\xi}$ as the case $M_{\xi}$ is analogous. Let  $I \in \mathcal I_{F^n(\xi)}$. Since the stable and unstable foliations are $F$-invariant and $F$ preserves intervals on center leaves we have $( F^{-n}(I) )_{\xi_i}  = F^{-n}(I_{F^n(\xi_i)})$
then
\begin{eqnarray*}
 F_*^n m_{\xi}(I)  =   m_\xi (F^{-n}(I)) & = &  \liminf_{m \rightarrow \infty} \frac{1}{m} \sum_{i=1}^m m_{\xi_i}( ( F^{-n}(I) )_{\xi_i}  )\\
 & = & \liminf_{m \rightarrow \infty} \frac{1}{m} \sum_{i=1}^m m_{\xi_i}(F^{-n}(I_{F^n(\xi_i)}))\\
  & = &  \liminf_{m \rightarrow \infty} \frac{1}{m} \sum_{i=1}^m \lambda^{-n} m_{\xi_i}(I_{F^n(\xi_i)})\\
  & = & \lambda^{-n} m_{F^n(\xi)}(I).
\end{eqnarray*}
 \end{proof}

Now let us see that the center Lyapunov exponent is defined everywhere and equals $log (\lambda)$. Using the above two lemmas we have that for a given $n$ 
\begin{eqnarray*}
\gamma^{-1} F^n_* \lambda_\xi (I) &  = &  F^n_* (\gamma^{-1}\lambda_\xi) (I) \leq F^n_* (m_\xi) (I)\\
& \leq & \lambda^{-n} m_{F^n(\xi)}(I) \leq \lambda^{-n} M_{F^n(\xi)}(I) \leq \gamma F^n_* \lambda_\xi (I).
\end{eqnarray*}

Dividing it all by $F^n_* \lambda_\xi (I)$ and applying $\frac{1}{n} log$, then
$$\frac{1}{n}log(\gamma^{-1})  \leq \frac{1}{n}log \left( \lambda^{-n} \frac{m_{F^n(\xi)}(I)}{F^n_* \lambda_\xi (I)}\right) \leq  \frac{1}{n}log\left( \lambda^{-n} \frac{ M_{F^n(\xi)}(I)}{F^n_* \lambda_\xi (I)} \right)\leq \frac{1}{n}log(\gamma).$$

Now let we shrink the interval $I$. Let $x\in I$ and define $I_\epsilon$ as the interval inside the center foliation of radius $\epsilon$ centered on $x$. Therefore

\begin{eqnarray}\label{eq:cal.exp}
 \lim_{\epsilon \rightarrow 0} \lambda^{-n} \frac{m_{F^n(\xi)}(I_\epsilon)}{F^n_* \lambda_\xi (I_\epsilon)} & = &\lim_{\epsilon \rightarrow 0} \lambda^{-n} \frac{m_{F^n(\xi)}(I_\epsilon)}{\lambda_\xi (F^{-n}(I_\epsilon))}\\
 & = & \lim_{\epsilon \rightarrow 0} \lambda^{-n} \frac{\int_{I_\epsilon} \rho_{F^n(\xi)} d\lambda_{F^n(\xi)}}{\lambda_\xi (F^{-n}(I_\epsilon))}.\nonumber
\end{eqnarray}

But we know that $m_{F^n(\xi)} (=\rho_{F^n(\xi)} \lambda_{F^n(\xi)})$ has bounded density with respect to the Lebesgue measure, then for some universal constant $\beta$

$$ \beta^{-1} \int_{I_\epsilon} d\lambda_{F^n(\xi)} \leq \int_{I_\epsilon} \rho_{F^n(\xi)} d\lambda_{F^n(\xi)} \leq  \beta \int_{I_\epsilon}  d\lambda_{F^n(\xi)}.$$

The two above estimates imply that 
\begin{eqnarray*}
  \beta^{-1} \lambda^{-n} ||DF^{-n}|\til{E^c}(x)||^{-1} & \leq & \lim_{\epsilon \rightarrow 0} \lambda^{-n} \frac{\int_{I_\epsilon} \rho_{F^n(\xi)} d\lambda_{F^n(\xi)}}{\lambda_\xi (F^{-n}(I_\epsilon))}\\ & \leq & \beta \lambda^{-n} ||DF^{-n}|\til{E^c}(x)||^{-1}.
\end{eqnarray*}

Hence applying $\frac{1}{n}log$ and passing to the limit as $n$ goes to infinity we get

\begin{eqnarray*}
 -log \lambda - \lim_{n \rightarrow \infty}\frac{1}{n}||DF^{-n}|\til{E^c}(x) || & \leq & \lim_{n\rightarrow \infty} \lim_{\epsilon \rightarrow 0} \lambda^{-n} \frac{\int_{I_\epsilon} \rho_{F^n(\xi)} d\lambda_{F^n(\xi)}}{\lambda_\xi (F^{-n}(I_\epsilon))}\\ & \leq  & -log \lambda - \lim_{n \rightarrow \infty}\frac{1}{n}||DF^{-n}|\til{E^c}(x) ||.
\end{eqnarray*}

From equation (\ref{eq:cal.exp}) take $I$ to be $I_\epsilon$ as defined above and passing to the limit as $n$ goes to infinity we get

$$0 \leq -log \lambda - \lim_{n \rightarrow \infty}\frac{1}{n}||DF^{-n}|\til{E^c}(x) || \leq 0.$$

This implies that
\begin{eqnarray}\label{eq:lyapunov.center}
\lim_{n \rightarrow \infty}\frac{1}{-n}||DF^{-n}|\til{E^c}(x) || = log \lambda, \; \forall x \in \mathbb R^3.
\end{eqnarray}

The above means that the center Lyapunov exponent of $F^{-1}$ equals $log(\lambda^{-1})$ for every point. In particular because $\pi:\mathbb R^3 \rightarrow \mathbb T^3$ is a local isometry and $\pi \circ F^{-1} = f^{-1} \circ \pi$ the center Lyapunov exponent of $f^{-1}:\mathbb T^3 \rightarrow \mathbb T^3$ is defined for every point in $\mathbb T^3$ and equals $log(\lambda^{-1})$.

\begin{lem}
 The diffeomorphism $f$ is in fact an Anosov diffeomorphism.
\end{lem}
\begin{proof}
We know that $\lim_{n \rightarrow \infty} ||Df^n|E^c_f||=log(\lambda), \forall x \in \mathbb T^3$. Consider $\varepsilon >0$ satisfying $\lambda_\varepsilon := \lambda - \varepsilon > 0$. Since the center exponent exists for every $x$ then, given $x \in \mathbb T^3$, there are $n_x \in \mathbb N$ and a neighborhood $\mathcal U_x$ of $x$ such that $\forall x \in \mathcal U_x$ $|Df^{n_x}|E^c| \geq e^{n_x \lambda_\varepsilon}$. Since $\mathbb T^3$ is a compact manifold take a finite cover $\mathcal U_{x_1} \ldots \mathcal U_{x_l}$. Let $C_i<1$ be small enough so that for $x \in \mathcal U_{x_i}$ then $|Df^{n}(x)|E^c| \geq C_{x_i} e^{n \lambda_\varepsilon}$ for all $n \in \{ 0, 1, \ldots, n_{x_i}\}$. Let $C:= \min_i \; C_{x_i}$, we then have that $|Df^{n}(x)|E^c| \geq C e^{n \lambda_\varepsilon}$ for all $x \in \mathbb T^3$ and $n \in \mathbb N$. Hence the center foliation is expanding, therefore $f$ is in fact an Anosov diffeomorphism.
 \end{proof}

\begin{lem}
 $f$ is $C^\infty$ conjugate to its linearization.
\end{lem}
\begin{proof}
Let us first prove that
 $$\lambda_f^s(x) = \lambda_A^s;\; \lambda_f^c(x) = \lambda_A^c;\; \lambda_f^u(x) = \lambda_A^u; \; \; \forall x \in Per(f),$$
 where $\lambda^*_{g}$ is the Lyapunov exponent in the direction $* \in \{s,c,u\}$ of the map $g \in \{f, A\}$. From Equation (\ref{eq:lyapunov.center}) we know that $\lambda_f^c(x) = \lambda_A^c$ on any periodic point $x \in \mathbb T^3$ for $f$. Since the center foliation is absolutely continuous and $f$ is an Anosov diffeomorphism, then by \cite{gogolev} we get that $\lambda_f^u(x) = \lambda_A^u$. Since $f$ and $A$ are volume preserving then $\lambda_f^s(x) + \lambda_f^c(x)+\lambda_f^u(x)=0$ and $\lambda_A^s(x) + \lambda_A^c(x)+\lambda_A^u(x)=0$, which implies that $\lambda_f^s(x) = \lambda_A^s$ for all periodic point $x \in \mathbb T^3$ of $f$. Hence if the Lyapunov exponents are constant on periodic points for any direction, then $f$ is $C^{1}$ conjugate to its linearization by \cite[Proposition 1.1]{varao.etds}.

Recall that $f= h^{-1}\circ A \circ h$. Take $h_0:\mathbb T^3 \rightarrow \mathbb T^3$ a $C^\infty$ diffeomorphism $C^1$ close to $h$. Note that $$h_0 \circ f \circ h_0^{-1} = h_0 \circ h^{-1} \circ A \circ h \circ h_0^{-1} = (h \circ h_0^{-1})^{-1} \circ A \circ (h\circ h_0^{-1}).$$


Then $h_0 \circ f \circ h_0^{-1}$ is $C^1$ close to $A$ since $(h \circ h_0^{-1})^{-1}$ and $(h\circ h_0^{-1})$ are close to the identity.  Observe that $h\circ h_0$ is a $C^1$ conjugacy between $h_0 \circ f \circ h_0^{-1}$ and $A$, then it naturally satisfies the smooth conjugacy hypothesis from \cite{bootstrap}, therefore we obtain that $h_0 \circ f \circ h_0^{-1}$ is $C^\infty$ conjugate to $A$. Now because $h_0$ and $h\circ h_0^{-1}$ are $C^\infty$ then $h$ is $C^\infty$ as we wanted to show. 
 \end{proof}

We have proven above that UBD property implies $C^\infty$ conjugacy. It is easy to see that the converse is true. The theorem is now proven.
 \end{proof}

\begin{rem}\label{rem}
It is worth to note that the construction of the measures $\eta_x$ (Lemma \ref{lemma:eta}) on the universal cover are used simply as a tool to obtain information on the Lyapunov exponent. Although the measures $\eta_x$ are intimately connect with the Rohklin disintegration of volume on the center foliations, they lack a very important property, they do not \textit{a priori} vary measurably. In our case we did not need to chose them in order to vary in such a way. Besides that, these measures really live on the universal cover, i.e. we cannot project them coherently on $\mathbb T^3$: two measures $\eta_x$ and $\eta_y$ in $\tilc_x$ and $\tilc_y$ respectively and such that $\pi(x)=\pi(y)$, then most likely satisfies $\pi_*\eta_x\neq \pi_*\eta_y$. But because we know that in the end $f$ is smoothly conjugate to its linearization we can induce the length measure from the linearization to the dynamics of $f$ and obtain such dynamically defined measures varying measurably. Hence, an interesting \textbf{program to tackle rigidity problems} related to invariant foliations seems to better understand partially hyperbolic dynamics which admits a measurably dynamically defined measures (e.g. $f_*\mu_x=\lambda \mu_{f(x)}$). With this generality one should not expect to obtain smooth rigidity results a priori, but one should start by obtaining rigidity results from the measurable point of view. That would most likely give implications on Lyapunov exponent, entropy, metric isomorphisms and so on.
\end{rem}

\textit{Acknowledgment:} The author would like to thank Andrey Gogolev for some comments concerning smooth conjugacy and Ali Tahzibi for some comments on the UBD property.

\bibliography{regisvarao.bib}
\bibliographystyle{ijmart}

IMECC-UNICAMP\\
Rua S\'ergio Buarque de Holanda, 651\\
Campinas, SP, Brazil, CEP 13083-859\\

\textit{Email address:} \textbf{regisvarao@ime.unicamp.br}

\end{document}